\documentclass[11pt]{amsart}
\usepackage{geometry}                
\usepackage{graphicx}
\usepackage{amssymb}
\usepackage{epstopdf}
\numberwithin{equation}{section}

\newcommand{\E}{\mathbb {E}}
\newcommand{\be}{\begin{equation}}
\newcommand{\ee}{\end{equation}}
\newcommand{\R}{\mathbb R}

\newcommand{\<}{\langle}
\renewcommand{\>}{\rangle}
\everymath{\displaystyle}
\newtheorem{thm}{Theorem}[section]
\newtheorem{lem}[thm]{Lemma}

\theoremstyle{definition}

\theoremstyle{remark}

\title{On The largest eigenvalue of products from the $\beta$-Laguerre ensemble}
\author{Zachary Gelbaum}
\address{Oregon State University}
\email{gelbaumz@math.oregonstate.edu}
\date{\today}                                           

\begin{document}
\begin{abstract}
We determine the limiting distribution of the largest eigenvalue of products from the $\beta$-Laguerre ensemble.  This limiting distribution is given by a Tracy-Widom law with parameter $\beta_0>0$ depending on the ratio of the parameters of the two matrices involved.\end{abstract}
\maketitle
\section{Introduction}
The limiting spectral behavior of products of random matrices has been the subject of a number of studies in random matrix theory and various results on the limiting spectral distribution of such products are by now known (e.g. \cite{MR2861673, MR2736204, MR2772389}).  In general the spectra of such products will be complex, but in the event it is real, e.g., that of the product of two Hermitian matrices where one is non-negative definite (see for example \cite{MR2292918, MR1370408, MR2293813}), it makes sense to speak of the largest eigenvalue.  There are strong limit laws known for these largest eigenvalues, but so far there are no results regarding the distribution of the fluctuations around the strong limit.  The purpose of this paper is to investigate this limiting distribution in the setting the $\beta$-Laguerre ensembles.  

The $\beta$-Laguerre ensemble generalizes the classical Laguerre ensemble by allowing $\beta$ to vary over the positive reals in 
\begin{equation}c_{n,\kappa}^\beta\prod_{i<j}|\lambda_i-\lambda_j|^\beta\prod_{k=1}^n\lambda_k^{\frac\beta2(\kappa-n+1)-1}e^{-\frac\beta 2\lambda_k}, \end{equation} where without loss of generality $\kappa\geq n$ and $c_{n,\kappa}^\beta$ is a normalizing constant (see e.g$.$ \cite{MR2641363}). The above densities first arose in the study of certain quantum systems and orthogonal polynomials (see \cite{MR2641363} and references therein), however there were initially no known random matrices with these eigenvalue densities.  Then in \cite{MR1936554} the authors constructed families of tridiagonal random matrices whos eigenvalue densities agreed with the above, and in \cite{MR2813333}  the limiting distribution of the largest eigenvalues was determined, thus generalizing the classical Tracy-Widom laws for $\beta=1,2,4$ to a family of distributions indexed by $\beta>0$, denoted $TW_\beta$.  

In a first approach to the general problem of finding the limiting distribution of the largest eigenvalue of a product of random matrices, we are free to choose which matrix ensemble to work with and the $\beta$-ensembles along the methods employed in \cite{MR2813333} are particularly amenable to such a study (the reader may note that throughout this paper we make the slight abuse of language in referring to both the above density and the corresponding family of random matrices as the $\beta$-Laguerre ensemble).  Our results are as follows:\begin{thm}Let $X_n^p$ and $X_n^q$ be two independent elements of the $\beta$-Laguerre ensemble, with $\kappa=p$ and $q$ respectively.  Assume that $n\leq p\leq q$ and that $p=O(n)=q$.  Then if $\lambda_{n,0}$ is the largest eigenvalue of $X_n^pX_n^q$, \[\frac{\lambda_{n,0}-\mu_n}{\sigma_n}\stackrel{d}\to TW_{\beta_0},\] where $TW_{\beta_0}$ denotes the Tracy-Widom Law with parameter $\beta_0$ and \[\beta_0=\lim_{n\to\infty}{C_n}\beta ,\quad \mu_n=(\sqrt n+\sqrt p)^2(\sqrt n+\sqrt q)^2,\quad \sigma_n=c_n\frac{(\sqrt n+\sqrt p)^\frac43(\sqrt n+\sqrt q)^\frac43}{(\sqrt{np})^\frac13(\sqrt{nq})^\frac13},\] the constants $C_n$ and $c_n$ being defined by (\ref{c_n}) and (\ref{C_n}) in section 2.4 below.\end{thm}

We have written the scaling terms to ease comparison to the case of a single matrix (e.g$.$ \cite{MR2813333}, Theorem 1.4), noting that $c_n\to c\in \R$ by the hypothesis $p=O(q)$.  It is worth noting that if both matrices are identically distributed, i.e$.$ $p=q$, then $C_n=2$, so even in the i.i.d$.$ case the parameter of the limiting Tracy-Widom law is different than that of the factors.

In \cite{MR2813333} the authors show how elements of the $\beta$-Laguerre ensembles can be realized as finite difference approximations to a stochastic differential operator on $[0,\infty)$.  Just as in the usual finite difference schemes, e.g., for the Laplacian on $[0,\infty)$, the lowest $k$ eigenvalues and eigenvectors converge to those of the limiting operator.  This characterization of the limiting distributions is robust and we make full use of the results and techniques in \cite{MR2813333} below, in particular section 5 in that paper.  We note here that although we assume in Theroem 1.1 that $n\leq p\leq q$, this is only to simplify the proof; one can relabel parameters without altering the arguments in any essential way. 

In the next section we outline the setup from \cite{MR2813333} and then proceed to the proof of Theorem 1.1.  We end with some remarks and further questions in section 3.

\section{Proof of Theorem 1.1}
\subsection{Tridiagonal elements of the $\beta$-Laguerre ensemble}  Here we briefly describe the tridiagonal matrix ensemble that realizes (1.1); for proofs and further discussion see \cite {MR1936554} and \cite{MR2641363}.  Let $\chi_\alpha$ denote the random variable with density \[\chi_\alpha\sim\frac2{\Gamma\left(\frac\alpha2\right)}x^{\alpha-1}e^{-x^2},\] said to be a chi random variable with parameter $\alpha$.  Let $B_n^\kappa$, $\kappa\geq n$ be the following matrix: \[B_n^\kappa=\begin{bmatrix}\tilde\chi_{\beta\kappa}\\\chi_{\beta(n-1)}&\tilde\chi_{\beta(k-1)}\\&\ddots&\ddots\\& &\chi_\beta&\tilde\chi_{\beta(\kappa-n+1)}
\end{bmatrix},\] where $\tilde\chi_\alpha$ and $\chi_\alpha$ denote independent chi random variables.  Then the eigenvalues of \[X_n^\kappa\equiv \left(B_n^\kappa\right)^*B_n^\kappa\] have density (1.1).  Note that $X_n^\kappa$ has \[\tilde\chi_{\beta(\kappa-j+1)}^2+\chi_{\beta(n-j)}^2\] along the main diagonal, $j=1,\dots, n$, and \[\tilde\chi_{\beta(\kappa-j)}\chi_{\beta(n-j)}\] above and below the main diagonal.

\subsection{Notation and Setup from \cite {MR2813333}}
Unless specified otherwise, for vectors $v,u\in \R^n$, $\<v,u\>$ denotes the Euclidean inner product and likewise for $\|v\|$.

Fix $\beta>0$ and let $X_n^i$, $i=p,q$, be as above.  Define \[H^p_n\equiv\frac{\mu_{n,p}-X_n^p}{\sigma_{n,p}},\qquad H^q_n\equiv\frac{\mu_{n,q}-X_n^q}{\sigma_{n,q}},\]  \[m_{n,i}\left(\frac{\sqrt{ni}}{\sqrt n+\sqrt i}\right)^\frac23=n^\frac13\left(\frac{\sqrt{\frac in}}{1+\sqrt{\frac in}}\right)^\frac23,\] \[\mu_{n,i}=(\sqrt n+\sqrt i)^2,\qquad \sigma_{n,i}=\frac{(\sqrt n+\sqrt i)^\frac43}{(\sqrt{ni})^\frac13}.\]  Note here that the $X_n^i$, and hence the $H_n^i$, are independent, a fact we will use repeatedly below. 


Let $L^*$ be the following subspace of $L^2$,\[L^*=\{f\in L^2[0,\infty):\,f(0)=0,\,\|f\|_*^2<\infty\}\] where \[\|f\|_*^2=\int_0^\infty (f^\prime)^2+xf^2+f^2dx.\]
Let $B$ be standard Brownian motion on $[0,\infty)$ and for $f\in L^*$ define \[H_\beta(f)=-\frac{d^2}{dx^2}f+xf+\frac2{\sqrt{\beta}} B^\prime f\] where $B^\prime f$ is the distribution given by \[\frac d{dt}\int_0^tf\,dB\] and where we denote the action of $H_\beta f$ on a test function $\phi\in C_c^\infty$ by \[(\phi,H_\beta f).\]  Thus if $\phi$ is a test function, \[(B^\prime f,\phi)=-(f^\prime B,\phi)-(fB,\phi^\prime).\]
In \cite{MR2813333} it is shown that $(g,H_\beta f)$ defines a continuous bilinear form on $L^*$ and if $\lambda$ denotes the smallest eigenvalue of $H_\beta$, given by \begin{equation}\label{eigenfunction definition}\lambda=\inf\{(f,H_\beta f),:\, f\in L^*,\,\|f\|_{L^2}=1\},\end{equation} then $-\lambda$ is distributed as $TW_\beta$: $-\lambda\sim TW_\beta$.  

Next let $L^*_{n,i}$ be the subspace of $L^2[0,\infty)$ consisting of step functions of the following form: \[f=\sum_{k=1}^nc_k\chi_{[\frac{k-1}{m_{n,i}},\frac k{m_{n,i}}]}.\]  Let $P_n$ be the projection from $L^2$ onto this subspace.  Then $L^*_{n,i}$ is isometric to $\R^n$ with the inner product \[m_{n,i}^{-1}\<v,u\>=m_{n,i}^{-1}\sum_{k=1}^nv_ku_k,\] \[\<f,g\>_{L^2}=\sum_{k=1}^nc_kd_km_{n,i}^{-1}=m_{n,i}^{-1}\<f,g\>_{\R^n}.\]  We let $T_n$ denote the shift operator \[(T_nv)_k=v_{k+1},\] that is, the operator given by the $n\times n$ matrix with $1$'s below the main diagonal and zero's elsewhere.  Then define the difference operator \[\Delta^i_nv_k=m_{n,i}(v_{k}-v_{k-1})=m_{n,i}(I-T_n^*)v_k,\]i.e., for $\phi\in C_c^\infty$ $\Delta^i_n\Delta_n^{i*}P_n\phi\to\phi''$ in $L^2$, and note $\|T_n\|=1$.  Additionally, for two vectors $u,v\in\R^n$ we denote by $u_\times v$ the vector \[(u_1v_1,\dots,u_nv_n).\]

$H_n^i$ now takes the following form: \[H_n^iv=-\Delta_n^i\Delta_n^{i*}v+\left(\Delta_n^iy_{n,1}^i\right)_\times v+\frac12\left(\Delta_n^iy_{n,2}^i\right)_\times T_nv+\frac12T^*_n\left(\Delta_n^iy_{n,2}^i\right)_\times v,\]
\[\Delta_n^iy^i_{n,j}=\eta_{n,j}^i+\Delta_n^iw_{n,j}^i,\]
\[(\eta_{n,1}^i)_k=\frac{m_{n,i}^2}{\sqrt{ni}}(n+i-\beta^{-1}\E[\tilde\chi^2_{\beta(i-k+1)}+\chi^2_{\beta(n-k)}])=\frac{m_{n,i}^2}{\sqrt{ni}}(2k-1)\]
\[(\eta_{n,2}^i)_k=\frac{m_{n,i}^2}{\sqrt{ni}}2(\sqrt{ni}-\beta^{-1}\E[\chi_{\beta(n-k)}\tilde\chi_{\beta(i-k)}]),\]
\[(w_{n,1}^i)_k=\frac{m_{n,i}}{\sqrt{ni}}\sum_{j=1}^k\left(n+i-\beta^{-1}(\chi_{n-j}^2+\tilde\chi_{i-j+1}^2)\right)-m_{n,i}^{-1}(\eta_{n,1}^i)_k\]
\[(w^i_{n,2})_k=\frac{m_{n,i}}{\sqrt{ni}}2\sum_{j=1}^k\left(\sqrt{ni}-\beta^{-1}\chi_{\beta(n-j)}\tilde\chi_{\beta(i-j)}\right)-m_{n,i}^{-1}(\eta^i_{n,2})_k.\]


We now collect some bounds we will need in the proof below.  In \cite{MR2813333} it is shown that for each $i$ and any subsequence $H^i_{n_m}$ there exists a further subsequence and a probability space such that the statements below hold almost surely and from now on we will assume we are working with such a subsequence. 

First we have that for any $\epsilon>0$ there is a $c^{i}_\epsilon>0$ such that \begin{equation}|\Delta_n^iw^i_{n,j,k}|\leq m_{n,i}\sqrt{\epsilon\tilde\eta_{n,k}^i+c_\epsilon^i}\end{equation} where \[\tilde\eta_{n,k}^i=\frac k{m_{n,i}}.\]  Next we have the following two bounds\begin{equation}\eta^i_{n,j,k}\leq2m_{n,i}^2,\qquad c^\eta_1\tilde\eta^i\leq\eta^i_{n,1,k}+\eta^i_{n,2,k}\leq c^\eta_2 \tilde\eta^i\end{equation} for some 
$c^\eta_i>0$.  Finally (cf section 6 in \cite{MR2813333}), there exist independent Brownian motions $B^i$ and processes $y^i_{j}(x)$ such that \begin{equation}y^i_{n,j}(x)\equiv(y^i_{n,j})_{\lfloor xm_{n,q}\rfloor}\mathbf {1}_{xm_{n,q}\in[0,n]}\to y^i_j(x)\end{equation} and \[y^i_{n,1}(x)+y^i_{n,2}(x)\to \frac2{\sqrt{\beta}}B^i+\frac {x^2}2\] in the Skorokhod topology on $D[0,\infty)$. 
 
\subsection{Outline of the proof}  Let $H_\beta^i$ denote the operator $H_\beta$ above with $B^i$ in place of $B$.  In \cite{MR2813333} the authors show, for each subsequence restricted to a further subsequence such that the above bounds hold a.s$.$, that the smallest eigenvalue and corresponding eigenvector of $H^i_n$ converge to that of $H_\beta^i$ using three Lemmas, numbered $5.6-5.8$, the content of which is as follows:  Lemma 5.6 states that there are positive constants $c^i_k$ independent of $n$ such that for all $v\in \R^n$ \[c^i_1\|v\|_{i,n*}^2-c^i_2m_{n,i}^{-1}\|v\|^2_2\leq m_{n,i}^{-1}\<H_n^iv,v\>_{\R^n} \leq c^i_3\|v\|_{n,i*}^2\] where \[\|v\|_{i,n*}^2= m_{n,i}^{-1}(\|\Delta^i_nv\|_{\R^n}^2+\|(\bar\eta^i_{n})_\times^\frac12v\|_{\R^n}^2+\|v\|_{\R^n}^2).\]  This is a coercivity bound used to control the eigenvectors as $n\to\infty$.  Lemma 5.7 establishes convergence in the sense of distributions, i.e., if $f_n\in L^*_{n,i}$ is such that $f_n\to f$ and $\Delta_n^if_n\to f'$ weakly in $L^2$ then for any $\phi\in C_c^\infty$ \[\<\phi,H_n^if_n\>_{L^2}\to(\phi,H_\beta^if).\]  Lastly Lemma 5.8 ensures that the eigenvectors of $H_n^i$ contain a subsequence converging to those of $H^i$: If $f_n\in L_{n,i}^*$, $\|f\|_{n,i*}^2\leq c<\infty$, and $\|f\|_{L^2}^2=1$ then there exists a subsequence $f_{n_k}$ such that $f_{n_k}\to_{L^2} f\in L^*$ and $\<\phi,H_{n_k}^if_{n_k}\>_{L^2}\to(\phi,H_\beta^if)$ for all $\phi\in C^\infty_c$.

We want to study the smallest eigenvalue of \begin{align}\notag H_n=\frac{\mu_{n,p}\mu_{n,q}I-X_n^pX_n^q}{\sigma_{n,p}^2\sigma_{n,q}^2}&=\frac{\mu_{n,p}I-X_n^p}{\sigma_{n,p}}\frac{X_n^q}{\sigma_{n,q}^2\sigma_{n,p}}+\frac{\mu_{n,p}}{\sigma_{n,p}^2\sigma_{n,q}}\frac{\mu_{n,q}I-X_n^q}{\sigma_{n,q}}\\\notag&= \frac{\mu_{n,q}}{\sigma_{n,q}^2\sigma_{n,p}}H_n^p(I-\frac{\sigma_{n,q}}{\mu_{n,q}} H_n^q)+\frac{\mu_{n,p}}{\sigma_{n,p}^2\sigma_{n,q}} H_n^q\\\notag&=a_n\bar H^p_n+b_n\bar H^q_n-\frac{m_{n,p}^2m_{n,q}^2}{m_n^4\sigma_{n,p}\sigma_{n,q}}\bar H_n^p\bar H_n^q \end{align} 
where \[\bar H_n^i=\frac{m^2_n}{m^2_{n,i}}H^i_n,\qquad m_n=\left(\frac{\left(\frac{\mu_q}{\sigma_{n,q}^2\sigma_{n,p}}m_{n,p}^2+\frac{\mu_p}{\sigma_{n,p}^2\sigma_{n,q}}m_{n,q}^2\right)m_{n,p}m_{n,q}}{\frac{\mu_q}{\sigma_{n,q}^2\sigma_{n,p}}m_{n,q}+\frac{\mu_p}{\sigma_{n,p}^2\sigma_{n,q}}m_{n,p}}\right)^\frac13\] and
\[a_n=\frac{m^2_{n,p}\mu_{n,q}}{m_{n}^2\sigma_{n,q}^2\sigma_{n,p}},\qquad b_n=\frac{m_{n,q}^2\mu_{n,p}}{m_n^2\sigma_{n,p}^2\sigma_{n,q}}.\]
This choice of $m_n$ ensures the proper scaling for the convergence we need below.

In the next section we determine the limiting operator of $ H_n$ in the sense above.  The product term $\bar H_n^p\bar H_n^q$ prevents us from directly applying Theorem 5.1 in \cite{MR2813333}, so instead we will follow the proof of that Theorem, stating and proving Lemmas analogous to those above. 

\subsection{Convergence}
To begin we first establish analogous almost sure bounds to those above.  We have \[\bar H^i_nv=-\Delta_n\Delta_n^*v+\left(\Delta_n\bar y^i_{n,1}\right)_\times v+\frac12\left(\Delta_n\bar y^i_{n,2}\right)_\times T_nv+\frac12T^*_n\left(\Delta_n\bar y^i_{n,2}\right)_\times v\]
where
\[\Delta_n=m_n(I-T_n^*),\]  
\[\Delta_n\bar y^i_{n,j}=\bar\eta^i_{n,j}+\Delta_n\bar w^i_{n,j},\]
\[\bar \eta^i_{n,j}=\frac{m^2_n}{m^2_{n,i}}\eta^i_{n,j},\qquad\bar w^i_{n,j}=\frac{m_{n}}{m_{n,i}} w^i_{n,j},\]i.e.,
\[(\bar y^i_{n,j})_k=\frac1{m_n}\sum_{i=1}^k(\bar\eta^i_{n,j})_k+(\bar w^i_{n,j})_k= \frac{m_n}{m_{n,i}}(y^i_{n,j})_k.\]
Noting that by hypothesis \[m_n=O(m_{n,p})=O(m_{n,q})=O(n^{1/3}),\]it follows easily from (2.2) and (2.3) that we can reduce to subsequences as above such that \begin{equation}\label{increment bound}|(\Delta_n \bar w^i_{n,j})_k|\leq m_n\sqrt{\epsilon\tilde\eta_{n,k}+c_\epsilon},\end{equation} 
\begin{equation}\label{eta bounds}\bar\eta^i_{n,j,k}\leq2m_{n}^2,\qquad c^\eta_1\tilde\eta\leq\bar\eta^i_{n,1,k}+\bar\eta^i_{n,2,k}\leq c^\eta_2 \tilde\eta,\end{equation} and the processes defined by \[\bar y^i_{n,j}(x)\equiv (\bar y^i_{n,j})_{\lfloor xm_{n}\rfloor}\mathbf {1}_{xm_{n}\in[0,n]}\] are convergent in the Skorokhod topology on $D[0,\infty)$, where $\tilde\eta_{n,k}= k/{m_n}$ and we reuse the notation for the constants from above, though they may be different here.  With the bounds (2.5)--(2.6) in hand, the proofs of Lemmas 5.6--5.8 in \cite{MR2813333} apply without change to $\bar H^i_n$, a fact we will use below.

If we now let $\bar y_{n,j}=\frac{a_n}{c_n}\bar y^p_{n,j}+\frac{b_n}{c_n}\bar y^q_{n,j}$ where 
\begin{equation}\label{c_n}c_n=a_n+b_n=\frac{(\sqrt{np}+\sqrt{nq})^2\left((\sqrt n +\sqrt q)^2\sqrt{np}+(\sqrt n+\sqrt p)^2\sqrt{nq}\right)}{(\sqrt n+\sqrt p)^4(\sqrt n+\sqrt q)^4},\end{equation} 
then by our choice of $m_n$ and using the independence of the $y^i$, it follows from \cite{MR2813333}, section 6, that there is a Brownian motion $B_x$ such that \[\bar y_{n,1}(x)+\bar y_{n,2}(x)\to\frac{x^2}{2}+\frac 2{\sqrt{C\beta}} B_x,\]
\begin{equation}\label{C_n}C=\lim_{n\to\infty}\left(\frac{m_n^3}{m_{n,p}^3}\frac{a_n^2}{c_n^2}+\frac{m_n^3}{m_{n,q}^3}\frac{b_n^2}{c_n^2}\right)^{-1}=1+\lim_{n\to\infty}\frac{p(\sqrt n+\sqrt p)^2+q(\sqrt n+\sqrt q)^2}{\sqrt{pq}\left((\sqrt n+\sqrt p)^2+(\sqrt n+\sqrt q)^2\right)},\end{equation}
in law with respect to the Skorokhod topology on $D[0,\infty).$  As already noted, we can reduce to a further subsequence such that this convergence holds almost surely on some probability space. We now have a candidate limiting operator:  \[ H_n\to{c}\left(-\frac{d^2}{dx^2}+x+ \frac2{\sqrt{C\beta}}B_x^\prime\right)=cH_{\beta_0},\qquad\beta_0=C{\beta},\,c=\lim c_n,\] the idea being that $c_n^{-1}(a_n\bar H_n^p+b_n\bar H_n^q)\to H_{\beta_0}$ and the product term $\bar H_n^p\bar H_n^q$ vanishes in the limit.  

In the following Lemma we let $L_n^*$ be the analogue of the discrete spaces already defined above for our new scaling term, e.g., $L^*_n$ is the space of step functions of the form \[f=\sum_{k=1}^nc_k\chi_{[\frac{k-1}{m_{n}},\frac k{m_{n}}]}\] and $P_n$ denotes the projection from $L^2$ onto this space.

\begin{lem} Let $f_n\in L^*_n$ be such that $f_n\to f$ and $\Delta_nf_n\to f'$ weakly in $L^2$.  Then for all $\phi\in C_c^\infty$ \[\<\phi, H_nf_n\>_{L^2}=\<P_n\phi, H_nf_n\>_{L^2}\to(\phi, cH_{\beta_0}f).\] 

\end{lem}
\begin{proof} The bounds (2.5)--(2.6) can be extended additively to $a_n\bar H_n^p+b_n\bar H_n^q$ and the proof of Lemma 5.7 in \cite{MR2813333} goes through without change to show that under the hypotheses above \begin{equation}\<\phi,(a_n\bar H_n^p+b_n\bar H_n^q)f_n\>_{L^2}\to(\phi,c H_{\beta_0}f).\end{equation} Next,\[\frac{m_{n,p}^2m_{n,q}^2}{m_n^4\sigma_{n,p}\sigma_{n,q}}=O(m_n^{-2}),\] so the proof of Lemma 2.1 reduces to showing \begin{equation}\notag m_n^{-2}\<\phi,\bar H_n^p\bar H_n^qf_n\>_{L^2}=m_n^{-2}\<\bar H_n^pP_n\phi,-\Delta_n\Delta_n^{*}f_n\>_{L^2}+m_n^{-2}\<\bar H_n^pP_n\phi,\bar H_n^qf_n+\Delta_n\Delta_n^{*}f_n\>_{L^2}\to0.\end{equation}

First note that for $g\in L^2$, $T_ng\to g$ in $L^2$ and likewise for $T_n^*$.  Then \[\<g,T_nf_n\>_{L^2}=\<T_n^*g,f_n\>_{L^2}\to\<g,f\>_{L^2}\] so $T_nf_n\to f$ weakly and likewise for $T_n^*f_n$.  Similarly $T_nT_n^*f_n\to f$ weakly.  Thus \[(T_n^*-I)(I-T_n)f_n\to0\] weakly.  Next observe that\[\<g,\Delta_n(T_n^*-I)(I-T_n)f_n\>=\<g,(I-T_n^*)(T_n-I)\Delta_n^*f_n\>=\<(T_n^*-I)(I-T_n)g,\Delta_n^*f_n\>\] and $(T_n^*-I)(I-T_n)g\to0$ in $L^2$.  We also have $\Delta_n^*f_n\to-f^\prime$ weakly.  Thus \[\Delta_n(T_n^*-I)(I-T_n)f_n\to0\] weakly as well and Lemma 5.7 now implies \[m_n^{-2}\<\bar H_n^pP_n\phi,-\Delta_n\Delta_n^*f_n\>_{L^2}=\<\phi,\bar H_n^p(T_n^*-I)(I-T_n)f_n\>_{L^2}\to0.\]

For the terms \[\<m_n^{-2}\bar H_n^pP_n\phi,\bar H_n^qf_n+\Delta_n\Delta_n^*f_n\>_{L^2}\] we note that from the proof of Lemma 5.7 in \cite{MR2813333} we have the following:  If $g_n\in L_n^*$ is such that $g_n$ is bounded both uniformly independent of $n$,  $g_n$ and $\Delta_ng_n$ both have supports that are contained in a finite interval $I$ for all $n$, and both are convergent in $L^2$ with \[g_n\stackrel{L^2}\to g\quad\mbox{and}\quad  \Delta_ng_n\stackrel{L^2}\to g^\prime,\] then \[\<g_n,\bar H_n^qf_n+\Delta_n\Delta_n^*f_n\>_{L^2}\to( g,\bar H^qf+f^{''})\] for all $f_n$ as above.  Thus if we show that $g_n=m_n^{-2}\bar H_n^pP_n\phi$ satisfies the above hypothesis and $g_n\to0$ the proof will be complete.

The existence of $I$ comes from $\phi\in C_c^\infty$ and uniform boundedness follows easily from (\ref{increment bound}) and (\ref{eta bounds}) together with the compact support and uniform boundedness of $P_n\phi$.

To control \[\Delta_nm_n^{-2}\bar H_n^pP_n\phi\] we first consider $\Delta_n(-m_n^{-2}\Delta_n\Delta_n^{*}P_n\phi)=(I-T_n^*)(T_n^*-I)\Delta_n^{*}P_n\phi.$  By the arguments above this converges to $0$ in $L^2$.  For the potential term \begin{align}&\Delta_nm_n^{-2}\left(\left(\Delta_n\bar y_{n,1}^p\right)_\times P_n\phi+\frac12\left(\Delta_n\bar y_{n,2}^p\right)_\times T_nP_n\phi+\frac12T^*_n\left(\Delta_n\bar y_{n,2}^p\right)_\times P_n\phi\right)\\\notag&\quad=(I-T_n^*)\left(\left((I-T_n^*)\bar y_{n,1}^p\right)_\times P_n\phi+\frac12\left((I-T_n^*)\bar y_{n,2}^p\right)_\times T_nP_n\phi\right.\\\notag&\left.\qquad+\frac12T^*_n\left((I-T_n^*)\bar y_{n,2}^p\right)_\times P_n\phi\right),\end{align} we note that $\bar y^p_{n,j}(x)$ are locally bounded and convergent a.e.  This combined with the compact support of $P_n\phi$ implies the $\bar y^p_{n,j}(x)$ converge locally in $L^2$, and by the arguments above regarding $T_n$ we find that the above converges to $0$ in $L^2$.  That $m_n^{-2}\bar H_n^pP_n\phi\stackrel{L^2}\to0$ follows similarly.

\end{proof}

\begin{lem}Define the following norm on $\R^n$: \[\|v\|_{*n}^2=m_n^{-1}(\|\Delta_nv\|_{\R^n}^2+\|(\tilde\eta_{n})_\times^\frac12v\|_{\R^n}^2+\|v\|_{\R^n}^2).\]  Then we have constants $C_k>0$ and $N>0$ such that for all $n>N$ \begin{equation} C_1\|v\|_{n*}^2-C_2m_n^{-\frac12}\|v\|_{\R^n}\sqrt{\|v\|_{n*}^2}-C_3m_n^{-1}\|v\|_{\R^n}^2\leq \< H_nv,v\>_{L^2}.\end{equation}

\end{lem}

\begin{proof} We have by definition
\begin{align}\bar H_n^iv=-\Delta_n\Delta_n^{*}\notag v&+\left(\left(\bar\eta_{n,1}^i\right)_\times v+\frac12\left(\bar\eta_{n,2}^i\right)_\times T_nv+\frac12T_n^*\left(\bar\eta_{n,2}^i\right)_\times v\right)\\&\notag+\left(\left(\Delta_n\bar w_{n,1}^i\right)_\times v+\frac12\left(\Delta_n\bar w^i_{n,2}\right)_\times T_nv+\frac12T_n^*\left(\Delta_n\bar w^i_{n,2}\right)_\times v\right)\\&=A^iv+B^iv+C^iv\notag.\end{align}

So letting \[d_n=\frac{m_{n,p}^2m_{n,q}^2}{m_n^4\sigma_{n,p}\sigma_{n,q}},\] we have \begin{align}\notag& a_n\<\bar H_n^pv,v\>+b_n\<\bar H_n^qv,v\>-\frac{m_{n,p}^2m_{n,q}^2}{m_n^4\sigma_{n,p}\sigma_{n,q}}\<\bar H_n^pv,\bar H_n^qv\>\\&=a_n(\<(A^p+B^p)(I-d_na_n^{-1}(A^q+B^q))v,v\>)+b_n\<(A^q+B^q)v,v\>\\&\quad+d_n\left(\<C^qv,(A^p+B^p)v\>+\<C^pv,(A^q+B^q)v\>+\<C^qv,C^pv\>\right)\\\notag&\qquad+a_n\<C^pv,v\>+b_n\<C^qv,v\>.
\end{align}
We first bound (2.13) and then (2.12).
We have from (\ref{increment bound}) \begin{align}m_n^{-1}\|\Delta_n^i\bar w_{n,j,k}v_k\|\leq\|\sqrt{\epsilon\tilde\eta_{n,k}+c_\epsilon}v_k\|.\end{align}

Then for $m_n^{-2}\<C^qv,C^pv\>$ we have \begin{align}\notag m_n^{-1}\|C^iv\|&\leq \|\sqrt{\epsilon\tilde\eta_{n,k}+c_\epsilon}v_k\|+\frac12\|\Delta_n^i\bar w_{n,2,k}^iT_nv_k\|+\frac12\|T_n^*\Delta_n^i\bar w_{n,2,k}^iv_k\|\\\notag&= \|\sqrt{\epsilon\tilde\eta_{n,k}+c_\epsilon}v_k\|+ \frac12\|\Delta_n^i\bar w_{n,2,k}^iv_{k+1}\|+\frac12\|T_n^*\Delta_n^i\bar w_{n,2,k}^iv_k\|\\\notag&\leq\|\sqrt{\epsilon\tilde\eta_{n,k}+c_\epsilon}v_k\|+ \frac12\|\sqrt{\epsilon\tilde\eta_{n,k}+c_\epsilon}v_{k+1}\|+\frac12\|T_n^*\Delta_n^i\bar w_{n,2,k}^iv_k\|\\\notag&\leq\|\sqrt{\epsilon\tilde\eta_{n,k}+c_\epsilon}v_k\|+ \frac12\|\sqrt{\epsilon\tilde\eta_{n,k+1}+c_\epsilon}v_{k+1}\|+\frac12\|\sqrt{\epsilon\tilde\eta_{n,k-1}+c_\epsilon}v_{k-1}\|\\\notag&\leq2\|\sqrt{\epsilon\tilde\eta_{n,k}+c_\epsilon}v_k\|
\end{align}
and so \begin{align}\notag|m_n^{-2}\<C^qv,C^pv\>|\leq4\|\sqrt{\epsilon\tilde\eta_{n,k}+c_\epsilon}v_k\|^2&=4\epsilon\|\sqrt{\tilde\eta_{n,k}}v_k\|^2+c_\epsilon\|v\|^2\\&\leq4\epsilon m_n\|v\|_{*n}^2+c_\epsilon\|v\|_{\R^n}^2.\end{align}

For the $\<A,C\>$ terms, \[m_n^{-1}\|A^iv\|=c^i\|(I-T_n^*)\Delta_n^{*}v\|\leq2c^i\|\Delta_nv\|\] for constants $c^i>0$, so we have \[m_n^{-1}\|A^iv\|\leq c_A\|\Delta_nv\|\]  for some $c_A>0$.  Thus \begin{align}\notag m_n^{-2}|\<C^qv,A^pv\>|&\leq2\|\sqrt{\epsilon\tilde\eta_{n,k}+c_\epsilon}v_k\|c_A\|\Delta_nv\|\\\notag&\leq 2c_A(\sqrt{\epsilon m_n\|v\|_{n*}^2}+\sqrt{c_\epsilon}\|v\|)\sqrt{m_n\|v\|_{n*}^2}\\&=2c_A\left(\sqrt{\epsilon}m_n\|v\|_{n*}^2+\sqrt{c_\epsilon}\|v\|_{\R^n}\sqrt{m_n\|v\|_{n*}^2}\right), \end{align}
and similarly for $m_n^{-2}\<A^qv,C^pv\>$.

For the $\<B,C\>$ terms note that \[m_n^{-2}|(\Delta_nw^i_{n,j})_k|\leq m_n^{-1}\sqrt{\epsilon\tilde\eta+c_\epsilon}\leq\sqrt\epsilon\sqrt{m_n^{-2}\tilde\eta}+m_{m,q}^{-1}\sqrt{c_\epsilon}\leq c_1\sqrt\epsilon+m_{m,q}^{-1}\sqrt{c_\epsilon}.\]  By Cauchy-Schwarz and (\ref{eta bounds}) we have \begin{align}\notag|m_n^{-2}\<C^qv,B^pv\>|&\leq c_2(c_1\sqrt\epsilon+m_n^{-1}\sqrt{c_\epsilon})\sum (\tilde\eta_{n})_kv_k^2\\&\leq c_3(c_1\sqrt\epsilon+m_n^{-1}\sqrt{c_\epsilon})m_n\|v\|_{n*}^2\end{align} and likewise for $m_n^{-2}\<C^pv,B^qv\>$.

For the remaining noise terms, we have from the proof of Lemma 5.6 in \cite{MR2813333} that \[\<C^iv,v\>\geq -c_4\sqrt{\epsilon} m_n\|v\|_{n*}^2-c_5(\epsilon)\|v\|_{\R^n}^2.\]

For (2.11), first we note that arguing as in \cite{MR2813333} using (\ref{eta bounds}) we have \[\<(A^p+B^p)v,v\>\geq0.\]   After some algebra we find\[\frac{d_n}{a_n}=\frac{\sqrt{\frac qn}}{(1+\sqrt{\frac qn})^2}m_{n}^{-2}\leq\frac14m_n^{-2}.\]

By definition, \begin{align}\notag m_n^{-2}\<(A^p+B^p)v,v\>&=\sum(m_n^{-2}(\bar\eta^p_{n,1})_k-2)v_k^2+m_n^{-2}(\bar\eta^p_{n,2})_k+2)v_{k}v_{k+1}\\\notag&\leq\sum (m_n^{-2}(\bar\eta^p_{n,2})_k+2)v_{k}v_{k+1}\\\notag&\leq 4\|v\|^2\end{align} using (\ref{eta bounds}) and Cauchy-Schwarz.  Thus \[d_na_n^{-1}\<(A^p+B^p)v,v\>\leq\|v\|^2\] and so \[I-d_na_n^{-1}(A^p+B^p)\] is Hermitian with spectrum contained in $[0,1]$.  Thus \[T\equiv(A^p+B^p)(I-d_na_n^{-1}(A^p+B^p)),\] being the product of two Hermitian, nonnegative matrices has only real, nonnegative eigenvalues (though it need not be normal).  Then using standard results (see e.g$.$ \cite{ MR1091716}, chapter 1 and \cite{MR1616464}) on the numerical range of $T$, \[\{\<Tv,v\>:\|v\|=1\},\] we see that $\<Tv,v\>\geq-\|v\|^2$.  Thus \begin{equation} \<(A^p+B^p)(I-d_na_n^{-1}(A^p+B^p))v,v\>\geq-\|v\|_{\R^n}^2.\end{equation}

Lastly, from \cite{MR2813333}, Lemma 5.6, we know \[\<(A^q+B^q)v,v\>\geq c_6m_n\|v\|_{n*}^2-c_7\|v\|^2.\]

Noting that $a_n$, $b_n$, and $d_n$  are convergent, we now have constants $c_8,c_9,c_{10}(\epsilon),c_{11}(\epsilon),c_{12}(\epsilon)>0$ such that \begin{align}&a_n\<\bar H_n^pv,v\>+b_n\<\bar H_n^qv,v\>-d_n\<\bar H_n^pv,\bar H_n^qv\>\\\notag&\qquad\geq (c_8-c_9O(\epsilon)-c_{10}(\epsilon)m_n^{-1})m_n\|v\|_{n*}^2-c_{11}(\epsilon)\|v\|\sqrt{m_n\|v\|_{n*}^2}-c_{12}(\epsilon)\|v\|^2\end{align} where $O(\epsilon)\to0$ as $\epsilon\to0$.  Taking $\epsilon$ small and then $n$ large establishes the Lemma.
\end{proof}
\begin{lem} Suppose $f_n\in L^*_n$ with $\|f_n\|^2_{*n}\leq c<\infty$ and $\|f_n\|_{L^2}=1$.  Then there exists $f\in L^*$ and a subsequence $f_{n_k}$ such that $f_{n_k}\stackrel{L^2}{\to}f$ and for all $\phi\in C_c^\infty$ we have \[\<\phi, H_{n_k}f_{n_k}\>_{L^2}\to(\phi,c H_{\beta_0}f).\]
\end{lem}
\begin{proof}
The proof is that same as that of Lemma 5.8 in \cite{MR2813333} and we omit it.
\end{proof}

Let $\bar\lambda_{n,0}$ and $v_{n,0}$ be the smallest eigenvalue and corresponding eigenvector of $ H_n$ such that $\|v_{n,0}\|^2_{L^2}=m_n^{-1}\|v_{n,0}\|^2_{\R^n}=1$, and let $\Lambda_0$ and $f_0$ be the same for $ H_{\beta_0}$.  To show that $\bar\lambda_{n,0}\to c\Lambda_0$ we can proceed exactly as in \cite{MR2813333}, repeating the arguments for completeness.

Suppose $\liminf \bar\lambda_{n,0}<\infty$.  Lemma 2.2 shows that $\bar\lambda_{n,0}$ is uniformly bounded below so there exists a subsequence such that $\bar\lambda_{n_k,0}\to\liminf\bar\lambda_{n,0}$.  Lemma 2.2 now implies that $\|v_{n_k,0}\|_{n*}^2$ are uniformly bounded, Lemma 2.3 then implies that a further subsequence converges to some $f\in L^*$ as in Lemma 2.1, and so Lemma 2.1 implies that for this further subsequence \[\<P_n\phi, H_{n_k}v_{n_k,0}\>_{L^2}\to (\phi,c H_{\beta_0}f).\] Then it follows that \[\frac{(\phi, cH_{\beta_0}f)}{\<f,f\>_{L^2}}=\liminf\bar\lambda_{n,0}\frac{\<\phi,f\>_{L^2}}{\<f,f\>_{L^2}}\] for all $\phi\in C_c^\infty$.  Thus \[\liminf\bar\lambda_{n,0}\geq c\Lambda_0.\]

To see $\limsup\bar\lambda_{n,0}\leq c\Lambda_0$, let $f^\epsilon\in C_c^\infty$ be such that $\|f^\epsilon-f_0\|^2_{*}<\epsilon$.  Then by the minmax principle and Lemma 2.1, \begin{align}\limsup\bar\lambda_{n,0}&\leq\limsup_{n\to\infty}\frac{\<P_nf^\epsilon, H_nP_nf^\epsilon\>_{L^2}}{\<P_nf^\epsilon,P_nf^\epsilon\>_{L^2}}\\\notag&=\frac{(f^\epsilon, cH_{\beta_0}f^\epsilon)}{\<f^\epsilon,f^\epsilon\>_{L^2}}.\end{align} Letting $\epsilon\to0$ we have \[\limsup\bar\lambda_{n,0}\leq\frac{(f_0,c H_{\beta_0}f_0)}{\<f_0,f_0\>_{L^2}}=c\Lambda_0.\]

Noting that by definition \[-\bar\lambda_{n,0}=c_n\frac{\lambda_{n,0}-\mu_n}{\sigma_n},\] what we have then is that for every subsequence of $\{\lambda_{n,0}\}$ there exists a probability space and a further subsequence along which \[\frac{\lambda_{n,0}-\mu_n}{\sigma_n}\to -\Lambda_0\]  almost surely.  Recalling that $-\Lambda_0\sim TW_{\beta_0}$, Theorem 1.1 obtains.

\section{Some remarks} The reader may note that contrary to the approach in the classical case, the framework in terms of a limiting operator allows us to avoid determining the eigenvalue densities for finite $n$, which, depending on one's point of view can be either an advantage or disadvantage to the approach.  

Although Theorem 1.1 does not tell us about the largest eigenvalue of the product of two independent Wishart matrices, it does suggest some interesting questions regarding the classical ensembles.  For example, in \cite{MR2293813} the authors determine the limiting empirical spectral distribution for a product of independent Wisharts, the limit depending on the ratio of the two parameters in the product.  The authors there conjecture that the limiting distribution of the largest eigenvalue of such a product is a Tracy-Widom law.  One can then ask the following:  If the limit does indeed follow a Tracy-Widom law $TW_\beta$, what is $\beta$, and does it depend on the parameters in a way similar to that in Theorem 1.1?  Much is still unknown about the full family of $TW_\beta$ distributions and it would be of interest to see them arise for $\beta\neq 1,2,4$ in the context of the classical ensembles.

\section{Acknowledgements} The author thanks his advisor, Harold Parks, for his time and encouragement, along with Yevgeniy Kovchegov and Mathew Titus for helpful discussions.
\bibliographystyle{plain}
\bibliography{products}

\begin{thebibliography}{10}

\bibitem{MR2292918}
Z.~D. Bai, Baiqi Miao, and Baisuo Jin.
\newblock On limit theorem for the eigenvalues of product of two random
  matrices.
\newblock {\em J. Multivariate Anal.}, 98(1):76--101, 2007.

\bibitem{MR2772389}
Charles Bordenave.
\newblock On the spectrum of sum and product of non-{H}ermitian random
  matrices.
\newblock {\em Electron. Commun. Probab.}, 16:104--113, 2011.

\bibitem{MR2293813}
J.-P. Bouchaud, L.~Laloux, M.~A. Miceli, and M.~Potters.
\newblock Large dimension forecasting models and random singular value spectra.
\newblock {\em Eur. Phys. J. B}, 55(2):201--207, 2007.

\bibitem{MR2736204}
Z.~Burda, R.~A. Janik, and B.~Waclaw.
\newblock Spectrum of the product of independent random {G}aussian matrices.
\newblock {\em Phys. Rev. E (3)}, 81(4):041132, 12, 2010.

\bibitem{MR1936554}
Ioana Dumitriu and Alan Edelman.
\newblock Matrix models for beta ensembles.
\newblock {\em J. Math. Phys.}, 43(11):5830--5847, 2002.

\bibitem{MR2641363}
P.~J. Forrester.
\newblock {\em Log-gases and random matrices}, volume~34 of {\em London
  Mathematical Society Monographs Series}.
\newblock Princeton University Press, Princeton, NJ, 2010.

\bibitem{MR1091716}
Roger~A. Horn and Charles~R. Johnson.
\newblock {\em Topics in matrix analysis}.
\newblock Cambridge University Press, Cambridge, 1991.

\bibitem{MR2861673}
Sean O'Rourke and Alexander Soshnikov.
\newblock Products of independent non-{H}ermitian random matrices.
\newblock {\em Electron. J. Probab.}, 16:no. 81, 2219--2245, 2011.

\bibitem{MR2813333}
Jos{\'e}~A. Ram{\'{\i}}rez, Brian Rider, and B{\'a}lint Vir{\'a}g.
\newblock Beta ensembles, stochastic {A}iry spectrum, and a diffusion.
\newblock {\em J. Amer. Math. Soc.}, 24(4):919--944, 2011.

\bibitem{MR1370408}
Jack~W. Silverstein.
\newblock Strong convergence of the empirical distribution of eigenvalues of
  large-dimensional random matrices.
\newblock {\em J. Multivariate Anal.}, 55(2):331--339, 1995.

\bibitem{MR1616464}
Pei~Yuan Wu.
\newblock A numerical range characterization of {J}ordan blocks.
\newblock {\em Linear and Multilinear Algebra}, 43(4):351--361, 1998.

\end{thebibliography}
\end{document}